\documentclass[12pt, reqno]{amsart}
\usepackage{amsmath, amsthm, amscd, amsfonts, amssymb, graphicx, color}
\usepackage[bookmarksnumbered, colorlinks, plainpages]{hyperref}
\hypersetup{colorlinks=true,linkcolor=red, anchorcolor=green, citecolor=cyan, urlcolor=red, filecolor=magenta, pdftoolbar=true}

\newtheorem{theorem}{Theorem}[section]
\newtheorem{lemma}[theorem]{Lemma}

\newtheorem{corollary}[theorem]{Corollary}
\theoremstyle{definition}
\newtheorem{definition}[theorem]{Definition}
\newtheorem{example}[theorem]{Example}

\theoremstyle{remark}
\newtheorem{remark}[theorem]{Remark}
\numberwithin{equation}{section}

\begin{document}

\setcounter{page}{1}

\title[Approximation of Entropy Numbers ]
 {Approximation of Entropy Numbers}
\author[K.P. Deepesh \MakeLowercase{and} V.B. Kiran Kumar]{K.P. Deepesh$^1$  \MakeLowercase{and}  V.B. Kiran Kumar$^2$}

\address{$^{1}$Department of Mathematics, \newline Government Brennen College, Thalassery, Kerala, India,  670 106.}
\email{\textcolor[rgb]{0.00,0.00,0.84}{deepesh.kp@gmail.com}}

\address{$^{2}$Department of Mathematics, \newline CUSAT, Kerala, India, 682 022}
\email{\textcolor[rgb]{0.00,0.00,0.84}{kiranbalu36@gmail.com}}


\let\thefootnote\relax\footnote{Copyright 2016 by the Tusi Mathematical Research Group.}


\date{Received: xxxxxx; Revised: yyyyyy; Accepted: zzzzzz.
\newline \indent $^{*}$K.P. Deepesh}

\begin{abstract}
The purpose of this article is to develop a technique to estimate certain bounds for entropy numbers
of diagonal operator on $\ell^p$ spaces for $1<p<\infty$ which improves the existing bounds. 
The approximation method we develop in this direction works for a very general class of operators between
Banach spaces, in particular reflexive spaces. 

As a consequence of this technique we also obtain the following results for a bounded linear operator $T$ 
between two separable Hilbert spaces:
\begin{equation*}
\epsilon_n(T)=\epsilon_n(T^*)=\epsilon_n(|T|) \; \text{for each}\; n\in \mathbb N,
\end{equation*}
where $\epsilon_n(T)$ is the $n^{th}$ entropy number of $T$.
This gives a complete answer to the question posed by B. Carl \cite{carl} in the setting of separable Hilbert spaces.
\end{abstract}
\maketitle

\section{Introduction and Preliminaries} 
\label{intro}
Let $X,Y$ be normed linear spaces and $T:X\rightarrow Y$ be a bounded linear operator. Let $U_X={\{x\in X:\|x\|\leq 1}\}$, the closed unit ball of $X$. Then the $n^{th}$ entropy number of $T$ is defined by
\begin{equation*}
  \epsilon_{n}(T):=\inf{\{\epsilon>0: \text{there exists}\; y_1,y_2,\dots,y_k\in Y\, \text{such that}\; T(U_X) \subseteq   \bigcup_{i=1}^{k}U(y_i, \epsilon),\; k\leq n}\}.
\end{equation*}
Here $U(y_0,r)={\{y\in Y:\|y-y_0\|\leq r}\}$ is the closed ball with center $y_0$ having radius $r>0$.

Among various measures of compactness, the entropy numbers are well known in studying the compactness of
bounded linear operators between Banach spaces. However, not much have been known about the
entropy numbers of general bounded linear operators between Banach spaces even though
numerous works have been made on estimating these quantities (see \cite{carl,gordon,kuhn,kuhn1}).

Note that the first entropy number of $T$ is equal to $\|T\|$ and $(\epsilon_k (T))$ is a non increasing sequence.
The entropy numbers of a bounded linear operator $T$ between Banach spaces measures the degree of compactness of $T$
and it is well known that
$$T \mbox{ is compact if and only if } \displaystyle \lim_{k\to \infty}\epsilon_k (T)=0.$$

The entropy numbers of an operator satisfy most of the properties of {\it s-numbers}
(See \cite{pie}). Hence in literature, entropy numbers are called pseudo s-numbers and
they play an important role in the theory of operator ideals and are also used in statistical learning theory.

 Consider the problem of approximating a quantity (in our case, the entropy numbers) assigned to $T\in BL(X, Y)$ (the space of all bounded linear operators between $X$ and $Y$), using a sequence of operators $T_n\in BL(X, Y)$ that converges to $T$ in some sense. There is a rich literature on such approximation techniques for the quantities like norm, eigenvalues, approximation numbers etc. Projection methods, and more generally finite section methods, mainly deal with
such kind of problems in which one tries to make use of finite rank operators or truncated
projections of an operator between infinite dimensional spaces to study about the operator.
%
It has been shown that the norms, eigenvalues, singular values and approximation numbers \cite{Bot,DKN,KVBR}, solutions of systems of linear equations etc. can be estimated by such kind of methods, under
various assumptions on the spaces or the operators and on the sense of convergence used.

Here we develop such a method to estimate the entropy numbers of operators acting between infinite dimensional Banach spaces.
To be more precise, for $T\in BL(X, Y)$, we consider the sequence $(T_n:= Q_nTP_n)$ of operators 
(usually each of them of finite rank) which converges to $T$ in the Strong Operator Topology (SOT), Weak or Weak$^{*}$ Operator Topology (WOT, WOT$^{*}$). We will show that the sequence $\epsilon_k(T_n)$ of entropy numbers of $T_n$ converges to $\epsilon_k(T)$ for each $k\in \mathbb{N}$, under some assumptions on the space $Y$ and on the operators $(Q_n)$ and $(P_n)$.

Also notice that in the setting of many Banach spaces with Schauder basis, we can choose the sequence $(T_n)$ to be the truncation of $T$ (that is $T_n=P_nTP_n$, where $P_n$'s are projections onto finite dimensional subspace spanned by first $n$ elements of the basis with $\|P_n\|\leq 1$ for each $n\in \mathbb N$). Hence it helps us to use the finite dimensional linear algebraic techniques in the computation of entropy numbers of an infinite dimensional operator. As an application of this techniques in the Banach space setting, we obtain a better estimate for the entropy numbers of certain diagonal operators between $\ell_p$ spaces, $1< p< \infty$.

As a second application of our main result we obtain the following result.
\begin{theorem}
Let $H_1,H_2$ be two complex separable Hilbert spaces and $T\in BL(H_1,H_2)$. Then
\begin{equation*}
\epsilon_n(T)=\epsilon_n(T^*)=\epsilon_n(|T|),
\end{equation*}
where $T^*$ denotes the Hilbert adjoint of $T$ and $|T|$ is the unique positive square root of $T^*T$.
\end{theorem}

We remark that the above result for compact operators on a  Hilbert space is proved by   D. E. Edmunds and R. M. Edmunds  \cite{EE}. In a more general setting, namely for bounded operators defined on a Hilbert space  using the polar decomposition was suggested by Brendt Carl in  \cite{Carlreview}); We also remark that Carl's technique works for bounded operators defined between two different Hilbert spaces.

The article is organized as follows. In the next section, we establish the approximation technique for entropy
numbers of a bounded linear operator between two Banach spaces, under the assumption that the co-domain is a reflexive
space. In the third section, we obtain a better estimate for the entropy numbers of certain diagonal operators between $\ell_p$ spaces, $1< p< \infty$, as an application of the approximation results to the non-Hilbert space settings. In the final section, we prove the above Theorem and also discuss some of the future possibilities of the considered problem at the end of the article.

\section{Approximation of Entropy Numbers of an Operator}\label{sec-2}

Let $S,T\in BL(X,Y)$. It is a well known property of the  entropy numbers that (see \cite{stephani} for eg.)
$$\epsilon_k(T+S)\leq \epsilon_k(T)+\|S\|,\; k\in \mathbb N.$$
It follows that

\[
\epsilon_k(T_n) = \epsilon_k(T_n - T +T)\leq \epsilon_k(T)+\|T_n - T\|.
\]

Replacing the role of $T_n$ and $T$ in the above inequality, we get
\[
\epsilon_k(T) =  \epsilon_k(T_n)+\|T - T_n\|.
\]

This implies that $\epsilon_k(T_n) \to \epsilon_k(T)$  as $n\to \infty$ if $ T_n \to T$  in the norm sense.
In other words, it says that the entropy number function $\epsilon_k$ (for a fixed $k$)
is a continuous real valued function on $BL(X,Y)$, with respect to the operator norm topology on $BL(X,Y)$.
In finite section methods and applications, we approximate $T$ by sequences of finite rank operators $(T_n)$ in $BL(X,Y)$.
Hence the convergence of $(T_n)$ to $T$ in the norm topology will force us to restrict our attention to compact operators. 
Therefore we are interested to approximate $\epsilon_k(T)$ by $\epsilon_k(T_n)$ when $(T_n)$ converges to $T$ in
a weaker senses of convergence (weaker than norm convergence).

In this regard, for $T\in BL(X, Y)$ and for each $n\in \mathbb{N}$, consider
$$T_n := Q_nTP_n \mbox{ where } Q_n \in BL(Y), P_n \in BL(X)
\mbox{ with }\|Q_n\|\|P_n\|\leq 1.$$
Then due to the
property $\epsilon_k(SRT)\leq \|S\|\epsilon_k(R)\|T\|$  (see \cite{stephani} for eg.), we have
$$\epsilon_k(T_n)\leq \epsilon_k(T)$$
and hence
$$\displaystyle\sup_{n}\epsilon_k(T_n)\leq \epsilon_k(T).$$
Thus if the limit of $\epsilon_k(T_n)$ exists, it can not be greater than $\epsilon_k(T)$.
To show the existence of the limit, we start with a simple lemma. First we define entropy numbers of a bounded subset of a metric space.

\begin{definition}
Let $S$ be a bounded subset of a metric space. For $n \in \mathbb{N}$, the $n$-th entropy number
of $S$ is the infimum of all $\epsilon >0$
for which there exists a cover for $S$ with $k$ closed balls, where $k\leq n$, having radii $\epsilon$.
That is,
$$\epsilon_n(S)= \inf\{\epsilon > 0 : \,\, S\subseteq \displaystyle\cup_{i=1}^k\, U(x_i, \epsilon),\quad k\leq n\}$$
The collection $\{x_1, x_2, \ldots, x_k\}\subseteq X$ is called an {\it $\epsilon$-net} for $S$.
Without loss of generality, we will assume that in an  {\it $\epsilon$-net} for $S$, each ball $U(x_i, \epsilon)$ intersects 
with $S.$
\end{definition}

Note that the $n^{th}$ entropy number $\epsilon_n(T)$ of $T\in BL(X,Y)$ is the $n^{th}$ entropy number of $T(U_X)$.
\begin{lemma}\label{lemma-1}
Let $A$ be a bounded non empty subset of a normed linear space $X$ and let
$M=\displaystyle\sup_{x\in A}\|x\|.$ Then there exists an $\epsilon-$net $\{x_1, x_2, \ldots, x_k\}\subset X$ for $A$
with $\|x_i\|\leq 2M$ for each $i=1, 2, \ldots, k$. In particular if $T: Y\to X$ is a bounded linear map then there exists
an $\epsilon-net$ $\{x_1, x_2, \ldots, x_k\}\subset X$ for $T(U_{Y})$, such that for each $i=1, 2, \ldots, k,
\|x_i\|\leq 2\|T\|.$
\end{lemma}
\begin{proof}
Since $\epsilon_k(A)\leq \epsilon_1(A) \leq M,$ there exists an $\epsilon-net$ $\{x_1, x_2, \ldots, x_k\}\subset X$ 
with $\epsilon\leq M$. If $\|x_i\|>2M$. Then $\|x_i-y\|> M$ for each $y \in A$.
Thus  $B(x_i, M) \cap A =\emptyset$. Therefore we have $\|x_i\|\leq 2M$ for all $i=1, 2, \ldots, k$.

In the particular case $A=T(U_{Y})$, we get the conclusion by taking $M=\|T\|$.
\end{proof}
Now we show that $\displaystyle\lim_{n\to \infty}\epsilon_k(T_n) = \epsilon_k(T)$
whenever $T_n= Q_nTP_n \to T$ in the Strong Operator Topology with $\|Q_n\|\|P_n\|\leq 1$,
provided that the co-domain space is reflexive.
\begin{theorem}\label{th-main}
 Let $X$ be a normed linear space, $Y$ be a reflexive Banach space and  $T\in BL(X,Y)$.
Let $(P_n)$ and $(Q_n)$ be sequences of operators in $BL(X)$ and $BL(Y)$ respectively
such that $\|Q_n\|\|P_n\|\leq 1$ for each $n\in \mathbb{N}$ and let $T_n:= Q_nTP_n$. 
If $T_n$ converges to $T$ in the pointwise sense of convergence, then for each
$k\in \mathbb{N}$, $\displaystyle\lim_{n\to \infty}\epsilon_k(T_n)= \epsilon_k(T).$
\end{theorem}

\begin{proof}
Fix $k\in \mathbb{N}$. Let us denote $d_n := \epsilon_k(T_n)$ and $d=\epsilon_k(T)$.\vskip0.15cm
Due to the norm assumptions on $Q_n$ and $P_n$ it follows that $d_n \leq d$ for all $n$, and hence $$lim\, sup\, d_n \leq d.$$

If $d=0$, the result is obvious. Hence we assume $d\neq 0$. Assume, if possible, that $\displaystyle\lim_{n\to \infty} d_n \neq d.$ There exists an $\epsilon >0$, and a subsequence $(d_m)_{m\in \mathbb{N}_1}$ of $(d_n)$
such that $$d_{m} <d-\epsilon \mbox{ for all } m\in \mathbb{N}_1,$$
where $\mathbb{N}_1$ is an infinite subset of $\mathbb{N}$.
Then for each $m\in \mathbb{N}_1$, $T_{m}(U)$ can be covered by $k$ (or less number of) balls  of radii $d-\epsilon.$
Without loss of generality, we assume that there are $k$ balls. Let the
corresponding $(d-\epsilon)$ - net be $\{y_{1}^{m}, y_{2}^{m}, \ldots, y_{k}^{m}\}$
for each $m\in \mathbb{N}_1$. That is,
$$T_m(U)\subseteq \displaystyle\cup_{i=1}^{k} B(y_i^{m}, d-\epsilon).$$

By Lemma \ref{lemma-1}, we can choose  $(y_i^m)$ with  $\|y_i^m\|\leq \|T_m\| \leq \|T\|$ for each 
$i\in \{1, 2, \ldots, k\}$. Therefore $(y_i^m)$ is a bounded sequence in $Y.$ 
Since $Y$ is a reflexive space and $(y_1^m)$ is a bounded sequence, by Eberlyn - Schmulyan theorem \cite{bvl},
$(y_1^m)$ has a weakly convergent subsequence, say $(y_1^{m_{j_1}})$ which converges to some $y_1$
in the weak sense of convergence.
Now since $(y_2^{m_{j_1}})$ is also bounded in $Y$, we get a subsequence $(y_2^{m_{j_2}})$
which is weakly convergent to some $y_2$. Taking subsequence of subsequences, we can find
an infinite set $\mathbb{N}_2\subseteq \mathbb{N}_1$ such that for each $i=1, 2, \ldots, k$,
$y_i^m \to y_i$ in the weak sense of convergence, as
$ m\to \infty\mbox{ in }\mathbb{N}_2.$

That is, for each $f\in Y^{\prime}$, there exists an $m(f)\in \mathbb{N}_2$ such that for each $i\in \{1, 2, \ldots, k\}$,
$$|f(y_i^m) - f(y_i)|<\frac{\epsilon}{4}, \mbox{ for all } m\geq m(f), m\in \mathbb{N}_2.$$
Now, we claim that $ T(U)\subseteq \displaystyle\cup_{i=1}^{k} B(y_i, d-\frac{\epsilon}{2}).$

To prove this claim, we take any $x\in U$ and $f\in Y^{\prime}$ such that $\|f\|\leq 1$.
Since $Tx \in T(U)$ and $T_m x \to Tx$ as $m\to \infty$ in $\mathbb{N}_2$, we can find an $m^*\in \mathbb{N}_2$ such that $m^* \geq m(f)$
and $$\|T_{m^*} x - Tx\|< \frac{\epsilon}{4}.$$
Since $T_{m^*}(U)$ is covered by $k$ balls of radii $d-\epsilon$, we have
$$\|T_{m^*}x- y_l^{m^*}\|<d-\epsilon,$$
for some $l\in \{1, 2, \ldots, k\} \mbox{ and } y_l^{m^*}\in Y$. Now let $y_l=\displaystyle\lim_{m\to \infty}y_l^m$. Then
\begin{eqnarray*}
|f(Tx) - f(y_l)| &\leq& |f(Tx) - f(T_{m^*}x)|\\
&+&|f (T_{m^*}x) - f(y_l^{m^*})|+ |f(y_l^{m^*}) - f(y_{l}) |\\
&<& \frac{\epsilon}{4} + d-\epsilon + \frac{\epsilon}{4}=d-\frac{\epsilon}{2},
\end{eqnarray*}
proving our claim by Hahn-Banach Theorem. But then $d=\epsilon_k(T)\leq d-\frac{\epsilon}{2}<d$,  a contradiction. Thus $\displaystyle\lim_{n\to \infty}d_n =d.$
\end{proof}

\begin{remark}\label{weakerassumption}
It can be observed that in the above theorem if $P_n$ and $Q_n$ are such that $T_n$ converges to $T$ weakly, then the same  proof  given above will work. Hence the theorem could be restated under the weaker assumption that $T_n \stackrel{WOT}{\longrightarrow} T$. Also, it can be seen that the reflexivity assumption on the co-domain is redundant, provided we assume the co-domain to be the dual space of some separable space. In particular, all the above mentioned conditions holds true when $X$ and $Y$ are separable Hilbert spaces.

\end{remark}
The analogue of Theorem \ref{th-main} with these weaker assumptions is given below.

\begin{corollary}\label{maincor}
Let $Y$ be the dual space of some separable space and $X, T, T_n$ be as in Theorem \ref{th-main}. If
$T_n \to T$ in the weak operator topology on $BL(X,Y)$, then
$$\displaystyle\lim_{n\to \infty}\epsilon_k(T_n)=\epsilon_k(T).$$
\end{corollary}
\begin{proof}
 The proof is exactly similar to that of Theorem \ref{th-main}, wherein we need to use the fact that $U_Y$ is weak* sequentially compact if $Y$ the dual space of a separable Banach space \cite{bvl}.
\end{proof}

The following example is an explicit illustration of Theorem \ref{th-main}.
\begin{example}
Consider the identity map $I : \ell_p \to \ell_p,\,\, 1< p <\infty$ over the field $\mathbb{R}$.
For each fixed $n\in \mathbb{N}$,
$$\epsilon_n(I) = 1.$$

Now consider the sequence of projection operators $(P_n)$ on $\ell_p$, which converges to the
identity operator $I$ point wise. It follows from \cite[Proposition 1.3.2]{stephani} that
$$1=\|P_n\| \geq \epsilon_k(P_n) \geq k^{-\frac{1}{2n}}.$$

Hence it follows that  $$\epsilon_k(P_n) \to  \epsilon_k(I) \mbox{ as } n\to \infty.$$
\end{example}

\begin{remark}
 It can be seen that if we do not assume $T_n = Q_nTP_n$ with the norm conditions, the conclusion is no longer true. For example, if we take $T_n$ on $\ell_2$ as $(T_n(x))(j)=0$ for $j\neq n$ and $(T_n(x))(n)=x_n$, for each $n\in \mathbb{N}$, then $T_n$ converges to $0$ in SOT; whereas $lim\, \epsilon_1(T_n)=\|T_n\|=1\neq 0= \epsilon_1(0).$
 
\end{remark}

\begin{remark}
If we choose $X=Y$ to be a Banach space with a Schauder basis and $P_n=Q_n$ to be the projections of norm $1$ onto the
finite dimensional subspace spanned by first $n$ elements of the Schauder basis of $X$, then the sequence $T_n$ shall
be identified as the first $n\times n$ block of the infinite matrix that represents $T$. The  theorem brings out the
possibility of using linear algebra techniques in computing $\epsilon_k(T)$.
\end{remark}

\section{A Banach space Application: A better Estimate for Entropy numbers}
Finding estimates for entropy numbers is a difficult task and it has been an interesting problem in the case of many concrete operators (See \cite{gordon}, \cite{kuhn},\cite{kuhn1}, \cite{schtt}). However, not much have been achieved in this regard to the best of our knowledge. In \cite{gordon}, the authors obtained certain estimates for entropy numbers of some special types of diagonal operators on real Banach spaces with a 1-unconditional basis (See Proposition 1.7 of \cite{gordon}). This result was modified for similar diagonal operators on $\ell_p$ spaces (real and complex), $1\leq p\leq \infty$ in \cite{stephani}.

Here, as an application to our approximation techniques to operators on Banach spaces, we
sharpen these estimates thereby strengthening Proposition 1.3.2 of \cite{stephani}.
The major ingredients in the proof
are the techniques used to prove Proposition 1.3.2 \cite{stephani}
and Theorem \ref{th-main}. 
\begin{theorem}\label{th-diag}
Let $\sigma_1 \geq \sigma_2 \geq \ldots \geq \sigma_k\geq  \ldots \geq 0$ and let $D$ be the operator defined on $\ell_p,\; (1< p < \infty)$ by
\begin{equation}\label{diagop}
D((\zeta_1, \zeta_2, \ldots, \zeta_k, \ldots ))=(\sigma_1 \zeta_1, \sigma_2 \zeta_2, \ldots, \sigma_k \zeta_k, \ldots ), \mbox{ for }
(\zeta_n)\in \ell_p,
\mbox{ and }1< p < \infty.
\end{equation}

 Then for each $n\in \mathbb{N}$,
$$\displaystyle\sup_{1\leq k <\infty} n^{-\frac{1}{k}}(\sigma_1\sigma_2\ldots\sigma_k)^{\frac{1}{k}}\leq \epsilon_n(D)
\leq 4\cdot \displaystyle\sup_{1\leq k <\infty} n^{-\frac{1}{k}}(\sigma_1\sigma_2\ldots\sigma_k)^{\frac{1}{k}}$$
in the case of real $\ell_p$ spaces and
$$\displaystyle\sup_{1\leq k <\infty} n^{-\frac{1}{2k}}(\sigma_1\sigma_2\ldots\sigma_k)^{\frac{1}{k}}\leq \epsilon_n(D)
\leq 4\cdot \displaystyle\sup_{1\leq k <\infty} n^{-\frac{1}{2k}}(\sigma_1\sigma_2\ldots\sigma_k)^{\frac{1}{k}}$$
in the case of complex $\ell_p$ spaces.
\end{theorem}
\begin{proof}
For $k\in \mathbb{N}$, define $D_k:= P_k D P_k,$ where $P_k$ are the standard $k^{th}$ projections
defined by $P_k(x)=(\zeta_1, \zeta_2, \ldots, \zeta_k, 0, 0, \ldots), \mbox{ for }
x=(\zeta_1, \zeta_2, \ldots, \zeta_j, \ldots )\in \ell_p,$
and let $\tilde{D}_k=D_k\mid_{R(P_k)}.$

Let $\epsilon > \epsilon_n(\tilde{D}_k)$. Then there exists $\{x_1, x_2, \ldots, x_n\}\subset R(P_k)$
such that $$\tilde{D}_k(U_{R(P_k)})\subseteq \cup_{i=1}^n (x_i+\epsilon \, U_{R(P_k)}).$$
 Now, since $ Vol\, (\tilde{D}_k(U_{R(P_k)}))=\sigma_1\sigma_2\ldots\sigma_k \, Vol\,(U_{R(P_k)}),$ we have by comparison of volumes,
$\sigma_1\sigma_2\ldots\sigma_k \, Vol\,(U_{R(P_k)})\leq n\,\, \epsilon^k \, Vol\,(U_{R(P_k)})$ which implies
$\epsilon \geq n^{-\frac{1}{k}}(\sigma_1\sigma_2\ldots\sigma_k)^{\frac{1}{k}}.$ Taking the limiting case, we get
$\epsilon_n(\tilde{D}_k)\geq n^{-\frac{1}{k}}(\sigma_1\sigma_2\ldots\sigma_k)^{\frac{1}{k}}.$

Now, since $\epsilon_n(\tilde{D}_k)\leq \epsilon_n(D)$, we get
$\epsilon_n(D)\geq n^{-\frac{1}{k}}(\sigma_1\sigma_2\ldots\sigma_k)^{\frac{1}{k}}.$
Taking supremum over $k\in \mathbb{N}$, we get
$\epsilon_n(D)\geq \displaystyle\sup_{1\leq k <\infty}n^{-\frac{1}{k}}(\sigma_1\sigma_2\ldots\sigma_k)^{\frac{1}{k}},$
which proves one side of the required inequality in the real case.

For the converse part, we define
$\delta(n)= \sup_{1\leq k <\infty}n^{-\frac{1}{k}}(\sigma_1\sigma_2\ldots\sigma_k)^{\frac{1}{k}}$
and claim that for each $n\in \mathbb{N}$ there exists an index $r$ with $\sigma_{r+1}\leq 2\,\delta(n)$.
So let $n\in \mathbb{N}$. Since $2^{m}\to \infty$ as $m\to \infty$,
there exists an $r$ with $n\leq 2^{r+1}$ and so $1\leq 2\, n^{-\frac{1}{r+1}}$. Then due to monotonicity of $(\sigma_i)$,

$ \sigma_{r+1} \leq (\sigma_1\sigma_2\ldots\sigma_{r+1})^{\frac{1}{r+1}}\leq 2\, n^{-\frac{1}{r+1}}(\sigma_1\sigma_2\ldots\sigma_{r+1})^{\frac{1}{r+1}}\leq \displaystyle 2 \sup_{1\leq k<\infty}  n^{-\frac{1}{k}}(\sigma_1\sigma_2\ldots\sigma_{k})^{\frac{1}{k}}.$

Thus $\sigma_{r+1} \leq 2\,\delta(n).$ Now suppose that $\sigma_1 \leq 2\,\delta(n)$. Then
$\epsilon_n(D)\leq \|D\|=\sigma_1 \leq 2\,\delta(n),$ which gives
$\sup_{1\leq k <\infty}n^{-\frac{1}{k}}(\sigma_1\sigma_2\ldots\sigma_k)^{\frac{1}{k}} \leq
\epsilon_n(D)\leq 2\, \sup_{1\leq k <\infty}n^{-\frac{1}{k}}(\sigma_1\sigma_2\ldots\sigma_k)^{\frac{1}{k}}.$
%
%

Now, if $\sigma_1 > 2\,\delta(n)$, there exists an $m$ with $\sigma_{m+1}\leq 2\, \delta(n) <\sigma_m.$
Now $D_m : \ell_p \to \ell_p$ is of rank $m$. Let $y_1, y_2, \ldots, y_N$ be a maximal system of elements
in $D_m(U)$ with $\|y_i-y_j\|> 4\,\delta(n)\mbox{ for } i\neq j.$
Then $D_m(U)\subseteq \cup_{i=1}^N (y_i+4\,\delta(n) U).$
Then $\epsilon_N(D_m)\leq 4\,\delta(n).$
Now using Theorem \ref{th-main} as $m\to \infty$, we get
$\epsilon_N(D)\leq 4\,\delta(n)=
4\,\sup_{1\leq k <\infty}n^{-\frac{1}{k}}(\sigma_1\sigma_2\ldots\sigma_k)^{\frac{1}{k}}.$
It can be shown that $n\geq N$ (see \cite{stephani}). Hence
$$\epsilon_n(D)\leq 4\cdot \sup_{1\leq k <\infty}n^{-\frac{1}{k}}(\sigma_1\sigma_2\ldots\sigma_k)^{\frac{1}{k}}.$$

The proof for the complex $\ell_p$ case follows in a similar way.
\end{proof}
\begin{remark}
Theorem \ref{th-diag} is valid for $p=1$ also, since $\ell_1$ is the dual of $c_0$
and therefore Corollary \ref{maincor} is applicable. For $\ell_{\infty}$, even though
it is the dual of $\ell_1$ and Corollary \ref{maincor} is applicable, we can not
obtain the estimate due to lack of useful projections.
\end{remark}

\section{ Entropy Numbers of Hilbert space Operators}

In this section we discuss about the entropy numbers of Hilbert space operators.
Let $H_1,H_2$ denote complex Hilbert spaces and $T:H_1\rightarrow H_2$ be a bounded linear operator with
$T^*$ denotes the adjoint of $T$. 

Note that $T$ has the polar decomposition $T=V|T|$, where $V$ is a partial isometry and $|T|$ is the unique positive square root of the operator $T^*T$. This decomposition is unique provided the null space of $V$ and the null space of $|T|$ are the same. We refer \cite{halmosproblembook} for more details.
The question of the identity $\epsilon_k(T)=\epsilon_k(T^*)$ was motivated by the problem posed by B. Carl,
whether $\epsilon_k(T^*)=O(n^{-\alpha})$ if $\epsilon_k(T)$ is so. Some positive anwers (partial) to this problem
are available in the literature \cite{EE,gordon}.
The following theorem was proved answering the problem for compact operators between Hilbert spaces.

\begin{theorem}\label{th-carl}(\cite{EE})
 Let$ H$ be a Hilbert space and let $T\in BL(H)$ be compact. Then for all $n\in N$, $$\epsilon_n(T)=\epsilon_n(T^*)=\epsilon_n(|T|).$$
\end{theorem}

The proof given in \cite{EE} makes use of the spectral theorem for compact operators between
Hilbert spaces. It is well known that every bounded linear operator between separable Hilbert spaces
can be approximated by compact operator in the strong sense of convergence, with the help of orthonormal projections.
This observation helps us to extend Theorem \ref{th-carl} to all bounded linear operators between seperable Hilbert space.

\begin{theorem}\label{edmund-karl}
Let $H_1,H_2$ be complex separable Hilbert spaces and  $T\in  BL(H_1,H_2)$. Then $\epsilon_n(T)=\epsilon_n(T^*)=\epsilon_n(|T|)$.
\end{theorem}
\begin{proof}
For $n\in \mathbb{N}$, let $P_n\in B(H_1)$
and $Q_n\in B(H_1)$ denote the $n^{th}$ standard orthonormal projections, which converges to the identity
operator on the respective spaces in the strong operator topology. Then,
$Q_nTP_n$ and $P_nT^*Q_n$ are compact operators satisfying
$$Q_nTP_n \stackrel{SOT}{\longrightarrow} T \mbox{ and }
P_nT^*Q_n \stackrel{SOT}{\longrightarrow} T^*, \mbox{ as } n \to \infty.$$
Also $\|P_n\|\leq 1, \|Q_n\|\leq 1.$
Hence by Theorem \ref{th-main}, for each $k\in \mathbb{N}$,
$$\epsilon_k(Q_nTP_n) {\longrightarrow} \epsilon_k(T) \mbox{ and }
\epsilon_k(P_nT^*Q_n) {\longrightarrow} \epsilon_k(T^*), \mbox{ as } n \to \infty.$$
Since $Q_nTP_n$ and $P_nT^*Q_n$ are compact operators, by Theorem \ref{th-carl} in \cite{EE},
$$\epsilon_k(Q_nTP_n)=\epsilon_k(P_nT^*Q_n)\mbox{ for each } k\in \mathbb{N}, n\in \mathbb{N},$$
we get the required conclusion, $\epsilon_k(T)=\epsilon_k(T^*)$.

Let $T=V|T|$ be the polar decomposition of $T.$ Note that $|T|=V^*T$.  Now
\[
\epsilon_k(P_n|T|P_n)=\epsilon_k(P_nV^*TP_n)\leq \epsilon_k(T), \textrm{ since }\|P_n\|\|V^{*}\|\|P_n\|\leq 1
\]
By taking limit $n \rightarrow\infty$, we have $\epsilon_k(|T|)\leq \epsilon_k(T).$

To prove the other way inequality, consider
\begin{align*}
\epsilon_k(P_nTQ_n)&=\epsilon_k(P_nV|T|Q_n) \leq \epsilon_k(|T|) 
\end{align*}
Taking limit as $n\rightarrow \infty$, we obtain $\epsilon_k(T)\leq \epsilon_k(|T|)$. Thus $\epsilon_k(T)=\epsilon_k(|T|)$ for each $k\in \mathbb N$.
\end{proof}
Note that, in the above theorem we have used the separability of Hilbert spaces to get finite rank projections which converges to the identity operator on the corresponding Hilbert space with respect to the strong operator topology. If the separability condition is dropped we may get finite rank projections but they  may not converge to the identity operator in the
strong operator topology. By a different approach a simple proof of Theorem \ref{edmund-karl} by dropping the assumption of separability of the space was suggested by B. Carl \cite{Carlreview}) in the review of the paper \cite{EE}.  We remark that the same technique works for operators defined between two different Hilbert spaces. In Theorem \ref{hilbert}, we state Carl's result for bounded linear operators defined between two different Hilbert spaces and for the sake of completeness, we produce the proof.

\begin{theorem}\label{hilbert}
Let $T\in BL(H_1,H_2)$. Then $\epsilon_k(T)=\epsilon_k(T^*)=\epsilon_k(|T|)$.
\end{theorem}
\begin{proof}
Let $T=V|T|$ be the polar decomposition of $T$. Then $\epsilon_{k}(T)\leq \|V\|\epsilon_k(|T|)=\epsilon(|T|)$ 
for each $k\in \mathbb N$. Also, as $|T|=V^*T$, the other inequality also holds true. Since $T^*=|T|V^*$, the equality $\epsilon_k(T^*)=\epsilon_k(|T|)$ holds true.
\end{proof}

\begin{remark}
We would like to make the remark that if one requires to prove the equality $\epsilon_k(T)=\epsilon_k(T^*)$ 
for operators between certain special Banach spaces (like $\ell_p, \,\, 1\leq p< \infty$, where one has
finite rank projections of norm $1$ which converges to identity operator in SOT/WOT), it is sufficient to
prove the result for compact operators or for finite rank operators, in view of Theorem \ref{th-main}. 
However, it is still an open problem to prove the equality $\epsilon_k(T)=\epsilon_k(T^*)$ for finite 
rank/compact operators between Banach spaces.
\end{remark}


\subsection{Some of the important problems}
\begin{enumerate}
\item Consider the original problem by Carl, it is still an open problem to show the identity $\epsilon_k(T)=\epsilon_k(T^*)$ for bounded operators on arbitrary Banach spaces. Even it is not clear whether $\epsilon_k(T^*)=O(n^{-\alpha})$ if $\epsilon_k(T)$ is so in the general setting. However, since our techniques are applicable for reflexive Banach spaces with suitable projections, it suffices to show these for compact operators on the such spaces where we are able to approximate bounded operators by sequence of compact operators (in WOT).
\item It is not known to us whether the conclusion of Theorem \ref{th-main} holds for operators whose co-domains are not isometric to dual of separable spaces. In other words, an example to show that the duality assumption on the co-domain space is mandatory will be worth finding.
\item  It is of interest to check whether the estimate obtained in Theorem \ref{th-diag} is strict or
can be improved further.
\item To study the behavior of entropy numbers under a random perturbation is another important problem.
\end{enumerate}

{\bf Acknowledgments.} 
The authors wish to thank Prof. M.N.N. Namboodiri of CUSAT, Dr. G. Ramesh and Dr. D. Venku Naidu of IIT Hyderabad, 
for their valuable suggestions. Also we are thankful to Kerala School Of Mathematics (KSOM), 
Kozhikode for the local hospitality during their visits.

\bibliographystyle{amsplain}

\end{document}